\documentclass[11pt,reqno]{amsart}

\usepackage[T1]{fontenc}
\usepackage{lmodern}
\usepackage{amsmath,amssymb,amsthm,mathtools}
\usepackage{microtype}
\usepackage[a4paper,margin=30mm]{geometry}
\usepackage[colorlinks=true,linkcolor=blue,citecolor=blue,urlcolor=blue]{hyperref}
\usepackage[nameinlink,noabbrev]{cleveref}

\numberwithin{equation}{section}

\newtheorem{theorem}{Theorem}[section]
\newtheorem{proposition}[theorem]{Proposition}
\newtheorem{lemma}[theorem]{Lemma}
\newtheorem{corollary}[theorem]{Corollary}
\theoremstyle{definition}

\theoremstyle{remark}
\newtheorem{remark}[theorem]{Remark}

\newcommand{\D}{\mathbb D}

\newcommand{\calH}{\mathcal{H}}
\newcommand{\HH}{\mathcal H}
\newcommand{\HD}{\mathcal H_{2}}
\newcommand{\Sp}{\mathcal S}
\newcommand{\ess}{\mathrm e}

\newcommand{\Mellin}{\mathcal M}

\title[The differential Hilbert operator between weighted Bergman spaces]
{The Differential Hilbert Operator Between Weighted Bergman Spaces}

\author{ Songxiao Li  and Mengmeng Zhou }

\address{Songxiao Li\\   Department of Mathematics, Shantou University, Shantou 515063, Guangdong, P.R. China.  }
\email{jyulsx@163.com}

 \thanks{Corresponding author: Songxiao Li.}

\address{Mengmeng Zhou\\   Department of Mathematics, Shantou University, Shantou 515063, Guangdong, P.R. China.  }
\email{25mmzhou@stu.edu.cn}

\date{}

\begin{document}

\begin{abstract}
In this paper, a complete characterization of the boundedness, compactness, norm and essential norm of the differential Hilbert operator $\mathcal{H}_2:A^2_\alpha\to A^2_\beta$  is obtained.   More precisely, $\mathcal H_2:A^2_\alpha \to A^2_\beta$ is bounded if and only if $-1<\alpha<0$ and $\beta\geq\alpha+2$ and it is compact if and only if
$-1<\alpha<0$ and $\beta>\alpha+2$.  The norm and essential norm of $ \|\mathcal H_2\|_{A^2_\alpha\to A^2_{\beta}}$ are also investigated. In particular, when $\beta=\alpha+2$,   
\[
 \|\mathcal H_2\|_{A^2_\alpha\to A^2_{\alpha+2}}  =\|\mathcal H_2\|_{\mathrm e,A^2_\alpha\to A^2_{\alpha+2}}
 =\frac{\pi\sqrt{(\alpha+2)(\alpha+3)}}
 {\sin(\pi(\alpha+2)/2)},\qquad -1<\alpha<0.
\]
When $\beta>\alpha+2$, $\|\mathcal H_2\|_{\mathrm e,A^2_\alpha\to A^2_{\beta}}=0$.  Furthermore, for every $1\leq p<\infty$,  the operator  
$\mathcal H_2$ belongs to the Schatten class $\mathcal S_p$   if and only if it is compact.


\end{abstract}

\maketitle

\section{Introduction}

 The Hilbert matrix is an infinite matrix $\mathcal{H}=( \frac1{n+k+1})_{n,k\ge 0}$. Let $f(z)=\sum_{k\geq0}a_kz^k$ denote an analytic function on the unit disk. The corresponding Hilbert operator is defined as 
\begin{equation}\label{eq:H-series-intro}
 \HH f(z)=\sum_{n=0}^\infty \left(\sum_{k=0}^\infty\frac{a_k}{n+k+1}\right)z^n
 =\int_0^1\frac{f(t)}{1-tz}\,dt,
\end{equation}
 for any $f$ for which the integral converges. The boundedness and norm of the Hilbert operator on Hardy spaces, Bergman spaces and weighted Bergman spaces have been extensively studied. Specifically,   the boundedness and norm of the Hilbert operator on Hardy spaces were studied in \cite{DiamantopoulosSiskakis2000,DostanicJevticVukotic2008}, where the exact norm on $H^p$ was determined to be $
\left \| \mathcal{H} \right \|_{H^p \to H^p}=\frac{\pi}{\sin\frac{\pi}{p}}
$
for $1<p<\infty$. The boundedness and norm of the Hilbert operator on Bergman spaces were studied in \cite{BozinKarapetrovic2018,Diamantopoulos2004,DostanicJevticVukotic2008}, where the norm on $A^p$ was established as
$
\left \| \mathcal{H} \right \|_{A^p\to A^p}=\frac{\pi}{\sin\frac{2\pi}{p}}
$
for $2< p<\infty$. For weighted Bergman spaces $A_{\alpha}^p$, the boundedness and norm of the Hilbert operator were studied in \cite{BaoTianWulan2026,Dai2024,JevticKarapetrovic2017,Karapetrovic2018,LindstromMiihkinenWikman2021,Norrbo2025}. In \cite{Karapetrovic2018}, Karapetrovi\'{c} showed that 
$\mathcal{H} :A_{\alpha}^p\to A^p_\alpha$ is bounded if and only if    $p>\alpha+2$, and he conjectured that the norm of the Hilbert operator on $A_{\alpha}^p$ is
\begin{equation*}
\left \| \mathcal{H} \right \|_{A_{\alpha}^p\to A^p_\alpha}=\frac{\pi}{\sin\frac{(\alpha+2)\pi}{p}}
\end{equation*}
when $p>\alpha+2$. Recently, it was proved in \cite{WulanZhouZhu2026} that for even exponents $p=2m$, 
\begin{equation*}
\left \| \mathcal{H} \right \|_{A_{\alpha}^{2m}\to A_{\alpha}^{2m}}=B(a,1-a),\ \ \ a=\frac{\alpha+2}{2m}, 
\end{equation*}
provided that $0<a\le m/(2m-1)$. This norm identity remains valid for all admissible parameters when $p=2,4,6,8,10$. However, Karapetrovi\'{c}'s conjecture is false for every even integer $p=2m$ with $m\ge 550000$. Related results on the essential norm can be found in \cite{LindstromMiihkinenNorrbo2022}.

The present paper  investigates the differential Hilbert operator $\calH_{2}$, which is defined by 
\begin{equation*} 
 \HD f(z)=\int_0^1\frac{f(t)}{(1-tz)^2}\,dt,
\end{equation*} 
 for any $f$ for which the integral converges. Let $\mathcal P$ stand for the space of analytic polynomials. The preceding expression initially defines a linear mapping $\HD$ on $\mathcal P$.  The term ``differential'' originates from the identity 
\begin{equation}\label{eq:factor-intro}
 \HD f(z)=z(\HH f)'(z)+\HH f(z)=(z\HH f(z))'.
\end{equation}  The
operator $\calH_{2}$ is of independent interest for the following reasons.

First, the differential Hilbert operator $\calH_{2}$   can be regarded as a  special case of the general Hilbert-type operator $\calH_{g}$ associated with the symbol $g(z)=z/(1-z)$. Let $H(\mathbb{D})$ be the space of all analytic  functions in the unit disk $\mathbb{D}  $ of the complex plane $\mathbb{C}$.   For any $g\in H(\D)$,  Galanopoulos, Girela, Peláez and Siskakis \cite{GGPS} defined and investigated the general Hilbert-type operator  
\[
\calH_{g}f(z)=\int_{0}^{1}f(t)\,g'(tz)\,dt,\qquad f\in H(\D).
\]
The classical Hilbert operator $\mathcal{H}$ is recovered by taking $g(z)=\log\frac1{1-z}$. Set $\tilde{g}(z)=\dfrac{z}{1-z}$. Direct computation yields $\tilde{g}'(z)=(1-z)^{-2}$, which implies
\[
\calH_{\tilde{g}}f(z)=\int_{0}^{1}\frac{f(t)}{(1-tz)^{2}}\,dt=\calH_{2}f(z).
\]
The work \cite{GGPS} establishes the boundedness properties of $\calH_{g}$ acting on Hardy spaces $H^p$, Bergman spaces $A^p_\alpha$, and Dirichlet-type spaces $D^p_\alpha$. Further conclusions concerning the   operator $\calH_{g}$ can be found in \cite{GGPS,PR,Tang2026}.

Second, the differential Hilbert operator has a close connection with the Bergman projection. The Bergman kernel on $\mathbb{D}$ admits the expression
$K(z,w) = \frac{1}{(1-\bar{w}z)^2},$ and the Bergman projection is given by
\[
P f(z) = \int_\mathbb{D} \frac{f(w)}{(1-\bar{w}z)^2} dA(w).
\]
One may observe that $\mathcal{H}_{2}$ and the Bergman projection $P$ possess the identical kernel $(1-\cdot z)^{-2}$. Their distinction lies solely in the domain of integration: the former uses $[0,1)$ equipped with the Lebesgue measure, while the latter is defined over $\mathbb{D}$ with the normalized area measure.

 Finally, the differential Hilbert operator constitutes a special case of the derivative-Hilbert operator $\textit{I}_{\mu_\alpha}$ proposed by Ye and Zhou (see \cite{ys3}),
\begin{equation}
\textit{I}_{\mu_\alpha} f(z)=\int_{0}^{1}\frac{f(t)}{(1-tz)^{\alpha}}\,d\mu(t), \nonumber
\end{equation}
where $\mu$ denotes a positive Borel measure defined on $[0,1)$. Ye and Zhou~\cite{YeZhou2022} provided a characterization of positive Borel measures $\mu$ for which the operator $\textit{I}_{\mu_\alpha}$ acts as a bounded or compact mapping between Bergman spaces, with the associated criteria formulated via Carleson-type properties of $\mu$. Setting $\alpha=\beta+1$ and $d\mu(t)=(1-t)^\beta dt$, the operator $\textit{I}_{\mu_\alpha}$ degenerates into the generalized Hilbert operator $\mathcal{H}_{\beta}$ introduced in \cite{li, ls}. Additional results concerning derivative-Hilbert operators and generalized Hilbert operators can be found in \cite{li, ls, YeFeng2023, ys3, YeZhou2022}.

Whether for the general Hilbert-type operator $\calH_{g}$ or the derivative-Hilbert operator $\textit{I}_{\mu_\alpha}$, characterizations of them on various function spaces are currently limited to qualitative analysis, and quantitative analysis cannot be achieved. Analogous to the Hilbert operator, several new methods can be developed to perform quantitative characterization for the operator $\mathcal{H}_2$.  In this work, we investigate the boundedness, compactness, operator norm, essential norm, and Schatten class membership of the operator $\mathcal{H}_2$ acting between weighted Bergman spaces. We prove that the operator $\HD:A^2_\alpha\to A^2_\beta$ is bounded precisely when
$
-1<\alpha<0 ~\text{and} ~ \beta\geq\alpha+2.
$
The operator $\HD:A^2_\alpha\to A^2_\beta$ is compact if and only if
$
-1<\alpha<0~\text{and}~ \beta>\alpha+2.
$
For the critical case $\beta=\alpha+2$, we establish
\begin{equation*} 
 \|\HD\|_{A^2_\alpha\to A^2_{\alpha+2}} =\|\HD\|_{\ess, A^2_\alpha\to A^2_{\alpha+2}}
 =\frac{\pi\sqrt{(\alpha+2)(\alpha+3)}}
 {\sin(\pi(\alpha+2)/2)}.
\end{equation*}
When $\beta>\alpha+2$, the essential norm satisfies $\|\HD\|_{\ess, A^2_\alpha\to A^2_{\beta}}=0$. In addition, for any $1\leq p<\infty$, the inclusion $\HD\in\Sp_p(A^2_\alpha,A^2_\beta)$ holds if and only if
$
-1<\alpha<0~\text{and}~ \beta>\alpha+2.
$

 The paper is organized as follows.  Section 2 collects all preliminary material and lemmas needed later. In
Section 3 we prove the   boundedness and compactness theorem, calculate the exact norm and essential norm when $\beta=\alpha+2$ and obtain an eigenvalue characterization of the operator norm when $\beta>\alpha+2$.  Finally, section 4 studies the Schatten classes and calculates the Hilbert--Schmidt norm. 

In this paper, we write   \(X\asymp Y\)  if there exists a   constant \(C>0\) satisfying \(X\leq CY\) and \(Y\leq CX\).

\section{Preliminaries and auxiliary lemmas}
\label{sec:prelim}

If $T:H_1\to H_2$ denotes a compact operator acting between Hilbert spaces, then $|T|=(T^*T)^{1/2}$ forms a positive, compact, self-adjoint operator on $H_1$. The eigenvalues of $|T|$ are referred to as the singular values of $T$, which are arranged in non-increasing order with multiplicities taken into account,
\[
 s_1(T)\geq s_2(T)\geq\cdots\geq0.
\]
For $0<p<\infty$, the Schatten class $\Sp_p(H_1,H_2)$   consists of all compact operators satisfying
\begin{equation*} 
 \|T\|_{\Sp_p}
 =\left(\sum_{j=1}^\infty s_j(T)^p\right)^{1/p}<\infty.
\end{equation*}
This expression defines a norm whenever $p\geq1$, whereas it yields a quasi-norm for $0<p<1$. The case $p=1$ corresponds to the trace class, and $p=2$ corresponds to the Hilbert--Schmidt class. In the Hilbert--Schmidt setting, for any orthonormal basis $\left \{ e_{k} \right \}_{k\ge 0} $ of $H_1$, the identity
\begin{equation*} 
 \|T\|_{\Sp_2}^2=\sum_{k=0}^\infty\|Te_k\|_{H_2}^2
\end{equation*}
holds.

Let $m$ denote the Lebesgue area measure on $\mathbb{D}$, i.e.,
$
dm(z)=dxdy=rdrd\theta,\  \ z=x+iy=re^{i\theta}.
$
For $0<p<\infty$ and $\gamma>-1$, the weighted Bergman space $A_{\gamma}^{p}$ is defined as $H(\mathbb{D}) \cap  L^{p}(\mathbb{D},dm_{\gamma}(z))$, where $dm_{\gamma}(z)=(\gamma+1)(1-|z|^{2})^{\gamma}dm(z)$. For $f\in A_{\gamma}^{p}$, its norm is given by
\begin{equation*}
\left \| f \right \|_{A_{\gamma}^{p}}=\left(\frac{\gamma+1}{\pi}\int_{\mathbb{D}} |f(z)|^p(1-|z|^2)^{\gamma}dm(z) \right)^{\frac{1}{p}}. 
\end{equation*}
In this paper, we restrict our attention to the case $p=2$.

\begin{lemma}[\cite{Zhu2007}]\label{lem:monomial}
For $\gamma>-1$ and $n\geq0$, set $ w_n(\gamma)=\|z^n\|_{A^2_\gamma}^2.$ Then
\begin{equation}\label{eq:wnexact}
 w_n(\gamma)
 =\frac{\Gamma(\gamma+2)\Gamma(n+1)}
 {\Gamma(n+\gamma+2)}
\end{equation}
and
\begin{equation}\label{eq:wn-asymp}
 w_n(\gamma)=\Gamma(\gamma+2)n^{-\gamma-1}
 \bigl(1+O(n^{-1})\bigr),\qquad n\to \infty.
\end{equation}
\end{lemma}

Using the above formula for $w_n(\gamma)$,  an orthonormal basis of $A^2_\gamma$ is given by
\begin{equation}\label{eq:orthonormal-basis}
 e_n^{(\gamma)}(z)=\frac{z^n}{\sqrt{w_n(\gamma)}},
 \qquad n=0,1,2,\ldots.
\end{equation}

For $\alpha$ and $\beta$ in $(0,\infty)$, the Beta function $B(\alpha,\beta)$ is defined by
\begin{equation*}
B(\alpha,\beta)=\int_{0}^{1}t^{\alpha-1}(1-t)^{\beta-1}\,dt.
\end{equation*}
If $0<\alpha<1$, then Euler's reflection formula gives
\begin{equation*}
B(\alpha,1-\alpha)=\frac{\pi}{\sin\pi \alpha}.
\end{equation*}

The boundedness and norm of the classical Hilbert operator acting on $A^2_\alpha$ are characterized by the following result, which can be found  in \cite{WulanZhouZhu2026}.

\begin{lemma} \label{lemma:Hnorm}
The Hilbert operator $\HH$ is bounded on $A^2_\alpha$ if and only if
$-1<\alpha<0$.  In this range,
\begin{equation*} 
 \|\HH\|_{A^2_\alpha\to A^2_\alpha}
 =B\left(\frac{\alpha+2}{2},1-\frac{\alpha+2}{2}\right)
 =\frac{\pi}{\sin(\pi(\alpha+2)/2)}.
\end{equation*}
\end{lemma}

For $-1/2<s<1/2$, define
\begin{equation}\label{eq:Ks}
 K_sF(x)=\int_0^\infty\frac{(y/x)^s}{x+y}F(y)\,dy,
 \qquad x>0.
\end{equation}

\begin{lemma}\label{lem:Ks-norm}
If $-1/2<s<1/2$, then $K_s$ is bounded on $L^2(0,\infty)$ and
\begin{equation}\label{eq:Ks-norm}
 \|K_s\|_{L^2(0,\infty)\to L^2(0,\infty)}
 =\frac{\pi}{\cos(\pi s)}.
\end{equation}
\end{lemma}

\begin{proof} We start our discussion on the dense subspace $C_c(0,\infty)$. Consider the Mellin transform
\[
(\mathcal{M}F)(\xi)=\frac1{\sqrt{2\pi}}
\int_0^\infty F(x)x^{-1/2-i\xi}\,dx.
\]
This operator acts as a unitary map from $L^2(0,\infty;dx)$ to $L^2(\mathbb R;d\xi)$. To establish this property, introduce the mapping
\[
V:L^2(0,\infty;dx)\longrightarrow L^2(\mathbb R;du),
\qquad
(VF)(u)=e^{u/2}F(e^u).
\]
Applying the change of variables $x=e^u$, we compute
\begin{equation*}
\|VF\|_{L^2(\mathbb R)}^2=\int_{-\infty}^{\infty}
e^u|F(e^u)|^2\,du =\int_0^\infty |F(x)|^2\,dx=\|F\|_{L^2(0,\infty)}^2.
\end{equation*}
The equality above verifies that $V$ is an isometry. For surjectivity, take an arbitrary $G\in L^2(\mathbb R)$ and set
\[
F(x)=x^{-1/2}G(\log x), \qquad x>0.
\]
Direct substitution yields
\begin{equation*}
\|F\|_{L^2(0,\infty)}^2=\int_0^\infty x^{-1}|G(\log x)|^2\,dx =\int_{-\infty}^{\infty}|G(u)|^2\,du=\|G\|_{L^2(\mathbb R)}^2
\end{equation*}
and
\[
(VF)(u)=e^{u/2}F(e^u)=G(u),
\]
which shows that $V$ is surjective. Hence $V$ is unitary.

 Using again the substitution $x=e^u$, we obtain
\begin{align*}
(\mathcal MF)(\xi)&=\frac{1}{\sqrt{2\pi}}\int_0^\infty F(x)x^{-1/2-i\xi}\,dx \\
&=\frac{1}{\sqrt{2\pi}}\int_{-\infty}^{\infty}e^{u/2}F(e^u)e^{-i\xi u}\,du=(\mathcal F VF)(\xi),
\end{align*}
where
\[
(\mathcal FG)(\xi)=\frac{1}{\sqrt{2\pi}}\int_{-\infty}^{\infty}G(u)e^{-i\xi u}\,du
\]
denotes the Fourier transform. By the Fourier--Plancherel theorem,
$\mathcal F$ is unitary on $L^2(\mathbb R)$. Consequently,
$\mathcal M=\mathcal F V$ is unitary from $L^2(0,\infty;dx)$ onto $L^2(\mathbb R;d\xi)$.

After the substitution $y=xt$ in \eqref{eq:Ks}, one has
\[
 K_sF(x)=\int_0^\infty\frac{t^s}{1+t}F(xt)\,dt.
\]
To justify the application of Fubini's theorem, it suffices to verify the absolute integrability of the integrand. For $F\in C_c(0,\infty)$, we have 
\begin{align*}
\int_0^\infty\int_0^\infty \left|\frac{t^s}{1+t}F(xt)x^{-1/2-i\xi}
\right|\,dx\,dt&=\int_0^\infty \frac{t^s}{1+t}\left(\int_0^\infty |F(xt)|x^{-1/2}\,dx\right)dt\\
& \overset{u=xt}{=} \left(\int_0^\infty \frac{t^{s-1/2}}{1+t}\,dt\right)
\left(\int_0^\infty |F(u)|u^{-1/2}\,du\right).
\end{align*}
Since $F\in C_c(0,\infty)$, the second factor is finite. By splitting the integral at $t=1$, we obtain 
\begin{equation*}
\int_0^\infty \frac{t^{s-1/2}}{1+t}\,dt=\int_0^1 \frac{t^{s-1/2}}{1+t}\,dt+\int_1^\infty \frac{t^{s-1/2}}{1+t}\,dt.
\end{equation*}
As $t\to0^+$ and $t\to\infty$, we have
\begin{equation*}
\frac{t^{s-1/2}}{1+t}\asymp t^{s-1/2}\qquad\text{and}\qquad\frac{t^{s-1/2}}{1+t}\asymp t^{s-3/2},
\end{equation*}
respectively. Hence
\begin{equation*}
\int_0^1 \frac{t^{s-1/2}}{1+t}\,dt<\infty\quad\Longleftrightarrow\quad
s> -\frac{1}{2},
\end{equation*}
and
\begin{equation*}
\int_1^\infty \frac{t^{s-1/2}}{1+t}\,dt<\infty\quad\Longleftrightarrow\quad
s<\frac{1}{2}.
\end{equation*}
So the integrand is absolutely integrable on $(0,\infty)\times(0,\infty)$. Consequently, applying $\Mellin$,  Fubini's theorem and then putting $u=xt$ in the inner integral gives
\begin{align}\label{eq:mellin}
 (\Mellin K_sF)(\xi)  &=\frac1{\sqrt{2\pi}}   \int_0^\infty\frac{t^s}{1+t}
   \left(\int_0^\infty F(xt)x^{-1/2-i\xi}\,dx\right)dt\nonumber \\
 &=\left(\int_0^\infty   \frac{t^{s-1/2+i\xi}}{1+t}\,dt\right)(\Mellin F)(\xi).
\end{align}
Thus $K_s$ becomes multiplication under $\Mellin$ by
\begin{align*}
 m_s(\xi) &=\int_0^\infty\frac{t^{s-1/2+i\xi}}{1+t}\,dt =B\left(s+\frac12+i\xi,\frac12-s-i\xi\right)\\
 &=\Gamma\left(s+\frac12+i\xi\right)   \Gamma\left(\frac12-s-i\xi\right) =\frac{\pi}{\cos(\pi(s+i\xi))}.
\end{align*}
The assumption $-1/2<s<1/2$ guarantees that the beta integral converges.
Since $\cos(a+ib)=\cos a\cosh b-i\sin a\sinh b$ and $\cosh^{2}b=1+\sinh^{2} b$, we can obtain 
\[
 |\cos(\pi(s+i\xi))|^2 =\cos^2(\pi s)+\sinh^2(\pi\xi).
\]
Since $\sinh^2(\pi\xi)\ge 0$ with equality if and only if $\xi=0$, it follows that
$|m_s(\xi)| \le \frac{\pi}{\cos(\pi s)}.$ Moreover, since $\cos(\pi s)>0$ when $-\frac{1}{2}<s<\frac{1}{2}$, we have 
\[ 
\operatorname*{ess\,sup}_{\xi\in\mathbb R}|m_s(\xi)|=|m_s(0)|=\frac{\pi}{\cos(\pi s)}.
\]
Let $M_{m_s}$ denote the multiplication operator $M_{m_s}g(\xi)=m_s(\xi)g(\xi)$ on $L^2(\mathbb R)$. Then we have
\[
\|M_{m_s}\|_{L^2(\mathbb R)\to L^2(\mathbb R)}=\operatorname*{ess\,sup}_{\xi\in\mathbb R}|m_s(\xi)|.
\]
Combining $\mathcal M$ is unitary and \eqref{eq:mellin}, we have $\mathcal M K_s F=M_{m_s}(\mathcal M F)$ for all $F\in C_c(0,\infty)$ and consequently
\[
\|K_sF\|_{L^2(0,\infty)}=\|M_{m_s}(\mathcal MF)\|_{L^2(\mathbb R)}.
\]
Hence
\[
\sup_{\substack{F\in C_c(0,\infty)\\F\neq0}}\frac{\|K_sF\|_{L^2(0,\infty)}}
{\|F\|_{L^2(0,\infty)}}=\|M_{m_s}\|_{L^2(\mathbb R)\to L^2(\mathbb R)}=\frac{\pi}{\cos(\pi s)}.
\]
Thus \eqref{eq:Ks-norm} holds on $C_c(0,\infty)$. Since $C_c(0,\infty)$ is dense in $L^2(0,\infty)$,  $K_s$ extends
uniquely to a bounded operator on $L^2(0,\infty)$ with the same norm. Therefore,
\[
\|K_s\|_{L^2(0,\infty)\to L^2(0,\infty)}=\frac{\pi}{\cos(\pi s)}.
\]
This completes the proof.
\end{proof}

Since $ \frac1{(1-tz)^2}=\sum_{n=0}^\infty(n+1)t^nz^n,$   for a polynomial $f(z)=\sum_ka_kz^k$, direct integration gives
\begin{equation}\label{eq:HD-coeff}
 \HD f(z)=\sum_{n=0}^\infty(n+1) \left(\sum_{k=0}^\infty\frac{a_k}{n+k+1}\right)z^n.
\end{equation}
It follows at once that \eqref{eq:factor-intro} holds on polynomials.

\begin{lemma}\label{lem:matrix}
Assume that $\mathcal{H}_2:A^2_\alpha \to A^2_\beta$ is bounded. Then relative to the orthonormal bases \eqref{eq:orthonormal-basis}, the coefficient matrix of the  operator $\HD: A^2_\alpha \to A^2_\beta $ is
\begin{equation}\label{eq:Bmatrix}
 B^{\alpha,\beta}=(b_{nk}^{\alpha,\beta})_{n,k\geq0}, \qquad
 b_{nk}^{\alpha,\beta} =\frac{n+1}{n+k+1} \sqrt{\frac{w_n(\beta)}{w_k(\alpha)}}.
\end{equation}
For $n,k\to\infty$ with $n/k$ in a fixed compact subset of $(0,\infty)$,
\begin{equation}\label{eq:B-asymptotic-general}
 b_{nk}^{\alpha,\beta}  =\sqrt{\frac{\Gamma(\beta+2)}{\Gamma(\alpha+2)}}
 \frac{n^{(1-\beta)/2}k^{(\alpha+1)/2}}{n+k} \bigl(1+o(1)\bigr).
\end{equation}
\end{lemma}

\begin{proof}
Fix $k\geq0$.  Formula \eqref{eq:HD-coeff} applied to the normalized monomial $e_k^{(\alpha)}=z^k/\sqrt{w_k(\alpha)}$ gives
\[
 \HD e_k^{(\alpha)}(z) =\sum_{n=0}^\infty   \frac{n+1}{n+k+1}   \frac{z^n}{\sqrt{w_k(\alpha)}}.
\]
Because $z^n=\sqrt{w_n(\beta)}e_n^{(\beta)}$, the coefficient of $e_n^{(\beta)}$ in this expansion is exactly
\[
 \frac{n+1}{n+k+1} \sqrt{\frac{w_n(\beta)}{w_k(\alpha)}}.
\]
This proves \eqref{eq:Bmatrix}. Next, by \eqref{eq:wn-asymp}, as $n,k\to\infty$,
\[
\sqrt{w_n(\beta)}=\sqrt{\Gamma(\beta+2)}\,n^{-(\beta+1)/2}\bigl(1+\mathcal{O}(n^{-1})\bigr),
\]
and it follows from $\frac{1}{1+x}=1-x+\mathcal{O}(x^2) $ holding as $x\to 0$ that 
\[
\frac{1}{\sqrt{w_k(\alpha)}}=\frac{k^{(\alpha+1)/2}}{\sqrt{\Gamma(\alpha+2)}}\bigl(1+\mathcal{O}(k^{-1})\bigr).
\]
Thus, because $\sqrt{1+x}=1+\mathcal{O}(x)$ holds as $x\to 0$, we obtain
\begin{equation*}
\frac{n+1}{n+k+1}\frac{\sqrt{w_n(\beta)}}{\sqrt{w_k(\alpha)}}=\sqrt{\frac{\Gamma(\beta+2)}{\Gamma(\alpha+2)}}
\frac{n+1}{n+k+1}\frac{k^{(\alpha+1)/2}}{n^{(\beta+1)/2}}\bigl(1+\mathcal{O}(n^{-1})+\mathcal{O}(k^{-1})\bigr).
\end{equation*}
If $n/k$ is restricted to a compact subset of $(0,\infty)$, then $n\asymp k$. Consequently,
\[
\mathcal{O}(n^{-1})+\mathcal{O}(k^{-1})=o(1)
\]
uniformly in this range. Since 
\begin{equation*}
\frac{n+1}{n+k+1}=\frac{n}{n+k}(1+o(1)),
\end{equation*}
we obtain \eqref{eq:B-asymptotic-general}.
\end{proof}

We next give a decomposition of $\mathcal{H}_{2}$. Let $g(z)=\sum_{n=0}^\infty c_nz^n$. Define
\begin{equation}\label{eq:Ddef}
 Dg(z)=(zg(z))'=\sum_{n=0}^\infty(n+1)c_nz^n .
\end{equation}
Then $\HD=D\HH$.  Unlike a general differentiation operator, $D$ is diagonal in the monomial bases, so its norm and compactness are immediate.

\begin{lemma}\label{lem:D}
Let $\alpha,\beta>-1$.  The operator $D:A^2_\alpha\to A^2_\beta$ is bounded if and only if $\beta\geq\alpha+2$.  In that case, 
\begin{equation*} 
 \|D\|_{A^2_\alpha\to A^2_\beta}  = \sup_{n\geq0}(n+1) \left[ \frac{\Gamma(\beta+2)\Gamma(n+\alpha+2)}
 {\Gamma(\alpha+2)\Gamma(n+\beta+2)} \right]^{1/2}. 
\end{equation*} 
Moreover, $D$ is compact if and only if $\beta>\alpha+2$. In particular,  
\begin{equation*}\label{eq:Dcritical}
 \|D\|_{A^2_\alpha\to A^2_{\alpha+2}} =\sqrt{(\alpha+2)(\alpha+3)}.
\end{equation*}
\end{lemma}

\begin{proof}
For the normalized monomial $e_n^{(\alpha)}$, we have
\begin{equation*}
 De_n^{(\alpha)} =(n+1)\frac{z^n}{\sqrt{w_n(\alpha)}} =(n+1)\sqrt{\frac{w_n(\beta)}{w_n(\alpha)}} e_n^{(\beta)}.
\end{equation*}
Thus the diagonal coefficient of $D$ between the orthonormal bases $\left \{ e_{n}^{(\alpha)} \right \}_{n\ge 0} $ and $\left \{ e_{n}^{(\beta)} \right \}_{n\ge 0} $ is
\begin{equation*}
 d_n=(n+1)\sqrt{\frac{w_n(\beta)}{w_n(\alpha)}}.
\end{equation*}
If $\left \{ d_n\right \}_{n\ge0} $ is bounded, then the diagonal operator $D:A_{\alpha}^{2}\to A_{\beta}^{2}$ is bounded and
\begin{equation*}
\|D\|_{A_{\alpha}^{2}\to A_{\beta}^{2}}\leq \sup_{n\geq0} d_n .
\end{equation*}
Conversely, assume that $D$ is bounded, then $\left \| De_{n}^{\alpha} \right \|_{A_{\beta}^{2}}= d_n $ gives 
\begin{equation*}
\sup_{n\geq0} d_n \leq \|D\|_{A_{\alpha}^{2}\to A_{\beta}^{2}},
\end{equation*}
so $D$ is bounded if and only if $ \left \{ d_n\right \}_{n\ge0} $ is bounded, and in that case $\left \| D \right \|_{A_{\alpha}^{2}\to A_{\beta}^{2}}=\sup_{n\geq0} d_n $. 

If $d_n\to0$, $\forall g(z)=\sum_{n=0}^{\infty}c_n e_n^{(\alpha)}(z)\in A_{\alpha}^2$ define
\begin{equation*}
D_N g(z)=\sum_{n=0}^{N}d_n c_n e_n^{(\beta)}(z).
\end{equation*}
Each $D_N$ has finite rank, and the diagonal norm formula gives
\begin{equation*}
\|D-D_N\|_{A_{\alpha}^{2}\to A_{\beta}^{2}}=\sup_{n>N} d_n  \rightarrow0.
\end{equation*}
Hence $D$ is compact.

Conversely, if $D$ is compact, then $e_n^{(\alpha)}\rightarrow 0$ weakly since $\{e_n^{(\alpha)}\}_{n\geq0}$ is orthonormal in $A_{\alpha}^2$. Therefore,
\begin{equation*}
 d_n =\|D e_n^{(\alpha)}\|_{A_{\beta}^2} \longrightarrow0.
\end{equation*}
Thus $D$ is compact if and only if $d_n\to0$.

 Using the  estimate
\begin{equation}\label{eq:standard gamma-quotient}
 \frac{\Gamma(n+a)}{\Gamma(n+b)} =n^{a-b}\bigl(1+O(n^{-1})\bigr),\qquad n\to \infty,
\end{equation}
we obtain
\begin{equation*}
 d_n =\sqrt{\frac{\Gamma(\beta+2)}{\Gamma(\alpha+2)}} n^{(\alpha-\beta+2)/2}(1+O(n^{-1})).
\end{equation*}
Since $\frac{\alpha-\beta+2}{2}\leq 0$ holds if and only if $\beta\geq \alpha+2$,
and $
\frac{\alpha-\beta+2}{2}<0$ holds if and only if $\beta>\alpha+2,$
the above characterization for boundedness and compactness yields that  $D$ is bounded precisely when $\beta\geq\alpha+2$, and compact precisely when $\beta>\alpha+2$.

Suppose that $\beta=\alpha+2$. Then
\begin{equation*} 
d_n^2=(\alpha+2)(\alpha+3)\frac{(n+1)^2}{(n+\alpha+2)(n+\alpha+3)}.
\end{equation*}
Since
\begin{equation*}
\frac{(n+1)^2}      {(n+\alpha+2)(n+\alpha+3)}\rightarrow 1\qquad (n\to\infty),
\end{equation*}
we have
\begin{equation*}
d_n \rightarrow \sqrt{(\alpha+2)(\alpha+3)}.
\end{equation*}
For any $  n\geq0$, we have
\begin{equation*}
(n+\alpha+2)(n+\alpha+3)-(n+1)^2=(2\alpha+3)n+\alpha^2+5\alpha+5.
\end{equation*}
Since $\alpha>-1$, we have $2\alpha+3>0$. In addition, as $\alpha^2+5\alpha+5$ is strictly increasing on $(-1,\infty)$, we obtain $\alpha^2+5\alpha+5>0$ and therefore $(n+\alpha+2)(n+\alpha+3)>(n+1)^2$. Consequently, $\forall n\geq0$
\begin{equation*}
d_n< \sqrt{(\alpha+2)(\alpha+3)}.
\end{equation*}
Together with the above limit, this yields
\begin{equation*}
\sup_{n\geq0} d_n=\sqrt{(\alpha+2)(\alpha+3)}.
\end{equation*}
Since $\|D\|_{A_{\alpha}^{2}\to A_{\alpha+2}^{2}}=\sup_{n\geq0}d_n$, it follows that
\begin{equation*}
\|D\|_{A_{\alpha}^{2}\to A_{\alpha+2}^{2}}=\sqrt{(\alpha+2)(\alpha+3)}.
\end{equation*}
This completes the proof.
\end{proof}

\section{Boundedness, compactness, norm, and essential norm}
\label{sec:norm-properties}

\begin{theorem}\label{thm:bdd} 
Let $\alpha,\beta>-1$.   Then the following statements hold. 
\begin{enumerate}
\item The operator $\HD:A^2_\alpha\to A^2_\beta$ is bounded if and only if
\begin{equation*} 
 -1<\alpha<0\quad\text{and}\quad \beta\geq\alpha+2.
\end{equation*}  
\item The operator $\HD:A^2_\alpha\to A^2_\beta$ is   compact if and only if
\begin{equation*} 
 -1<\alpha<0\quad\text{and}\quad \beta>\alpha+2.
\end{equation*}
\end{enumerate}
\end{theorem}

\begin{proof} $(1)$.  Suppose that $\HD:A^2_\alpha\to A^2_\beta$ is bounded. We first note that the evaluation functional at the origin $E_0(g)=g(0)$ is bounded on $A_\beta^2$ and satisfies $\|E_0\|=1$. For a polynomial $f(z)=\sum_{k=0}^Na_kz^k$, the constant term in \eqref{eq:HD-coeff} gives
\[
\mathcal H_2(f)(0)=\int_0^1 f(t)\,dt=\sum_{k=0}^{N}\frac{a_k}{k+1},
\]
the infinite sum is well defined since $a_k=0$ for $k>N$.  Put $x_k=a_k\sqrt{w_k(\alpha)}$.  Then
\begin{equation*}
 (\mathcal H_2f)(0)=\sum_{k=0}^N   \frac{x_k}{(k+1)\sqrt{w_k(\alpha)}}, \qquad \|f\|_{A^2_\alpha}^2=\sum_{k=0}^N|x_k|^2.
\end{equation*}
By the Riesz representation theorem, the functional $f\mapsto \mathcal H_2(f)(0)$ is bounded on $A_\alpha^2$ if and only if
\[
\sum_{k=0}^{\infty}\frac{1}{(k+1)^2w_k(\alpha)}<\infty,
\]
and its squared norm is given by
$
\sum_{k=0}^{\infty}\frac{1}{(k+1)^2w_k(\alpha)}.
$
By \eqref{eq:wn-asymp}, we obtain that for every  $ k\ge1$,
\[
\sum_{k=0}^{\infty} \frac{1}{(k+1)^2w_k(\alpha)}\asymp \sum_{k=0}^{\infty} k^{\alpha-1}<\infty\quad\Longleftrightarrow\quad \alpha<0.
\]

 Since the evaluation functional at the origin satisfies $\|E_0\|=1$, for every polynomial $f$ we have
\begin{equation*}
|\mathcal H_2(f)(0)|=|E_0(\mathcal H_2 f)|\leq
\|\mathcal H_2f\|_{A_\beta^2}\leq \|\mathcal H_2\|_{A_\alpha^2\to A_\beta^2} \|f\|_{A_\alpha^2}.
\end{equation*}
Thus the functional $f\mapsto\mathcal H_2(f)(0)$ is bounded on the polynomials with respect to the $A_\alpha^2$ norm. Since the polynomials
are dense in $A_\alpha^2$, it extends uniquely to a bounded linear functional on $A_\alpha^2$. By the preceding characterization, this is
possible only when $-1<\alpha<0$.

Let
\begin{equation*}
U_\alpha:A_\alpha^2\to\ell^2,\qquad U_\alpha\left(\sum_{k\ge 0}a_kz^k  \right)=(a_{k}\sqrt{w_k(\alpha)})_{k\ge 0}.
\end{equation*}
A simple proof shows that $U_{\alpha}$ is unitary. Similarly, one may define a unitary mapping $U_{\beta}:A_\beta^2 \to \ell^2$. Under these unitary identifications,
$\mathcal H_2$ is represented by the  matrix
$B^{\alpha,\beta}$ in \eqref{eq:Bmatrix}, that is, $B^{\alpha,\beta}=U_\beta\mathcal H_2U_\alpha^{-1}$. Therefore
\begin{equation*}
\|\mathcal H_2\|_{A_{\alpha}^{2}\to A_{\beta}^{2}}=\|B^{\alpha,\beta}\|_{\ell^2 \to \ell^2}.
\end{equation*}

For each integer $N\geq1$, define
\[
x_k^{(N)}=
\begin{cases}
N^{-1/2}, & N\leq k\leq 2N,\\
0, & \text{otherwise}.
\end{cases}
\]
It is clear that $x^{(N)}\in \ell^2$. Suppose now that $N\leq n,k\leq2N$, then $n/k\in[1/2,2]$,which is a compact subset of $(0,\infty)$. Hence, from the
uniform form of \eqref{eq:B-asymptotic-general} there exist constants
$c_{\alpha,\beta}>0$ and $N_0$ such that for all $N\geq N_0$,
\[
 b_{nk}^{\alpha,\beta}\geq c_{\alpha,\beta} N^{(\alpha-\beta)/2},\qquad N\leq k\leq2N.
\]
All entries and all nonzero coordinates of $x^{(N)}$ are positive, so for $N\leq n\leq2N$,
\begin{align*}
 (B^{\alpha,\beta}x^{(N)})_n  =\sum_{k=N}^{2N}b_{nk}^{\alpha,\beta}N^{-1/2} \geq c_{\alpha,\beta}(N+1) N^{(\alpha-\beta-1)/2}\geq c_{\alpha,\beta}N^{(\alpha-\beta+1)/2}.
\end{align*}
Therefore, 
\begin{equation}\label{eq:B-N}
 \|B^{\alpha,\beta}x^{(N)}\|_{\ell^2}
 \ge \left(\sum_{n=N}^{2N}\left| (B^{\alpha,\beta}x^{(N)})_n\right|^2\right)^{\frac{1}{2}}\ge  c_{\alpha,\beta}N^{1+(\alpha-\beta)/2}.
\end{equation}
If $B^{\alpha,\beta}$ is bounded, the exponent of $N$ must be nonpositive, that is, $1+\frac{\alpha-\beta}{2}\leq0,$ which is equivalent to $\beta\geq\alpha+2$. Consequently, a necessary condition for the boundedness of $\mathcal H_2:A_\alpha^2\to A_\beta^2$ is $-1<\alpha<0$ and $\beta\geq\alpha+2$.

Conversely, assume that $-1<\alpha<0$ and $\beta\geq\alpha+2$.  By Lemma \ref{lemma:Hnorm},
$ \HH:A^2_\alpha\longrightarrow A^2_\alpha$ is bounded.  By Lemma \ref{lem:D}, $ D:A^2_\alpha\longrightarrow A^2_\beta$ is bounded.  Therefore the operator $D\HH$ is bounded and
\[
 \|D\HH f\|_{A^2_\beta} \leq\|D\|_{A_{\alpha}^2 \to A_{\beta}^2}\,\|\HH\|_{A_{\alpha}^2 \to A_{\alpha}^2}\,\|f\|_{A^2_\alpha}.
\]
Therefore, $D\mathcal{H}$ is bounded. By \eqref{eq:Ddef},  we have $\mathcal{H}_2=D\mathcal{H}$ on the set of polynomials. Since the polynomials are dense in $A_{\alpha}^2$ and $D\mathcal{H}$ is bounded on polynomials, it follows that $\mathcal{H}_2=D\mathcal{H}$ on the whole space $A_{\alpha}^2$. Hence, $\mathcal{H}_2$ is bounded.

$(2)$. Suppose $-1<\alpha<0$  and $\beta>\alpha+2$. Then by Lemma \ref{lem:D}, we obtain that
$D:A^2_\alpha\to A^2_\beta$ is compact. Therefore, $\HD=D\HH$ is compact.

Conversely, assume that $\beta=\alpha+2$. Let $\left \{ y_k \right \}_{k\ge 0} \in\ell^2$. Then Cauchy--Schwarz yields
\[
 |\langle x^{(N)},y\rangle| \leq\|x^{(N)}\|_{\ell^2} \left(\sum_{k\geq N}|y_k|^2\right)^{1/2}\rightarrow0.
\]
Thus $x^{(N)}\rightharpoonup0$ in $\ell^2$.  At $\beta=\alpha+2$, it follows from \eqref{eq:B-N} that
\[
 \|B^{\alpha,\alpha+2}x^{(N)}\|_{\ell^2}\geq c_\alpha>0.
\]
This contradicts the compactness of $\mathcal{H}_2$. If $\alpha \notin (-1,0)$ or $\beta<\alpha+2$, then Theorem \ref{thm:bdd} implies that $\mathcal{H}_{2}:A_{\alpha}^2\to A_{\beta}^2$ is not bounded. Since every compact operator is bounded, it follows that $\mathcal{H}_{2}$  is not compact. This completes the proof of both conclusions.
\end{proof}

\begin{corollary}\label{cor:same-space} For every $\alpha>-1$, the differential Hilbert operator $\HD$ is unbounded on
$A^2_\alpha$.  Thus $\HD$ belongs to no Schatten ideal on $A^2_\alpha$.
\end{corollary}

\begin{proof} If $\HD$ were bounded from $A^2_\alpha$ into itself,   Theorem \ref{thm:bdd} would give both $-1<\alpha<0$ and
$\alpha\geq\alpha+2$.  The second inequality is impossible.  Thus the operator on $A^2_\alpha$ is unbounded.  Every Schatten-class operator is
bounded, so $\mathcal{H}_2$ cannot belong to any Schatten class.
\end{proof}

Fix $-1<\alpha<0$, and set 
\begin{equation*} 
 s=\frac{\alpha+1}{2}\in(0,1/2).
\end{equation*}
Then it follows from \eqref{eq:B-asymptotic-general} that  
\begin{equation}\label{eq:critical-matrix-asymp}
 b_{nk}^{\alpha,\alpha+2} =\sqrt{(\alpha+2)(\alpha+3)}\frac{(k/n)^s}{n+k}(1+o(1))
\end{equation}
when $n,k$ tend to infinity at comparable rates.

\begin{lemma}\label{lem:discrete-continuous} For $-1<\alpha<0$, we have 
\begin{equation*} 
 \|B^{\alpha,\alpha+2}\|_{\ell^2 \to \ell^2}\geq \frac{\pi\sqrt{(\alpha+2)(\alpha+3)}}{\sin(\pi(\alpha+2)/2)}.
\end{equation*}
Moreover, for every compact operator $C$ on $\ell^2$,
\begin{equation*} 
 \|B^{\alpha,\alpha+2}-C\|_{\ell^2 \to \ell^2}\geq \frac{\pi\sqrt{(\alpha+2)(\alpha+3)}}{\sin(\pi(\alpha+2)/2)}.
\end{equation*}
\end{lemma}

\begin{proof} Choose a nonzero $F\in C_c(0,\infty)$ and assume that its support is contained in $[a,b]$, where $0<a<b<\infty$.  For $N\geq1$, define
\begin{equation}\label{eq:JN}
 (J_NF)_k=N^{-1/2}F(k/N),\qquad k\geq0.
\end{equation}
Only the indices $aN\leq k\leq bN$  are nonzero. We obtain
\begin{equation}\label{eq:JNnorm}
 \|J_NF\|_{\ell^2}^2 =\frac1N\sum_{k\geq0}|F(k/N)|^2 \rightarrow \int_a^b|F(y)|^2\,dy=\int_0^\infty|F(y)|^2\,dy.
\end{equation}

 Fix $0<c<d<\infty$, take $x\in[c,d]$ and set $n_N=[Nx]$, then
\[
\left|\frac{n_N}{N}-x\right|\leq \frac{C}{N},
\]
for some constant $C>0$. Hence $\frac{n_N}{N}\rightarrow x$ uniformly for $x\in[c,d]$. Similarly,
$\frac{[Ny]}{N}\rightarrow y $ uniformly for $y\in[a,b]$.

For $(x,y)\in[c,d]\times[a,b]$, by the uniform asymptotic formula \eqref{eq:critical-matrix-asymp}, we have
\begin{align*}
N b_{n_N,[Ny]}^{\alpha,\alpha+2}&=N\sqrt{(\alpha+2)(\alpha+3)}\frac{\bigl([Ny]/n_N\bigr)^s}{n_N+[Ny]}\bigl(1+o(1)\bigr)\\
&=\sqrt{(\alpha+2)(\alpha+3)}\frac{\left(\dfrac{[Ny]/N}{n_N/N}\right)^s}{\dfrac{n_N}{N}+\dfrac{[Ny]}{N}}\bigl(1+o(1)\bigr).
\end{align*}
Letting $N\to \infty$. Using the uniform convergence established above, we obtain
\begin{equation}\label{eq:q_N(x,y)}
N b_{n_N,[Ny]}^{\alpha,\alpha+2}\rightarrow\sqrt{(\alpha+2)(\alpha+3)}\frac{(y/x)^s}{x+y}
\end{equation}
uniformly on $[c,d]\times[a,b]$. Using \eqref{eq:JN}, we have
\begin{align*}
N^{1/2}(B^{\alpha,\alpha+2} J_NF)_{n_N} =\sum_{k\geq0}b_{n_Nk}^{\alpha,\alpha+2}F(k/N)  =\frac1N\sum_{k\geq0} \bigl(N b_{n_Nk}^{\alpha,\alpha+2}\bigr)F(k/N).
\end{align*}
Since $F$ is supported in $[a,b]$, the last expression is an ordinary Riemann sum. That is, by \eqref{eq:q_N(x,y)}, it converges uniformly for $x\in[c,d]$ to 
\begin{align*}
\frac1N\sum_{k\geq0}\bigl(N b_{n_Nk}^{\alpha,\alpha+2}\bigr)F(k/N)&\to \sqrt{(\alpha+2)(\alpha+3)} \int_{a}^{b}\frac{(y/x)^s}{x+y}F(y)\, dy\\
&=\sqrt{(\alpha+2)(\alpha+3)} K_sF(x).
\end{align*}
Hence
\begin{equation*}
N^{1/2}(B^{\alpha,\alpha+2} J_NF)_{n_{N}}\rightarrow\sqrt{(\alpha+2)(\alpha+3)} K_sF(x)
\end{equation*}
uniformly for $x\in[c,d]$.

Keeping only the indices $[Nc]\leq n\leq[Nd]$ and using the preceding uniform convergence, we obtain
\begin{align*}
\liminf_{N\to\infty}\|B^{\alpha,\alpha+2} J_NF\|_{\ell^2}^2
&\geq\lim_{N\to\infty}\frac1N\sum_{n=[Nc]}^{[Nd]}\left|N^{1/2}(B^{\alpha,\alpha+2} J_NF)_n\right|^2 \\
&=(\alpha+2)(\alpha+3)\int_c^d |K_sF(x)|^2\,dx.
\end{align*}
Finally, letting $c\downarrow0$ and $d\uparrow\infty$ and applying the monotone convergence theorem gives
\[
\liminf_{N\to\infty}\|B^{\alpha,\alpha+2} J_NF\|_{\ell^2}\geq\sqrt{(\alpha+2)(\alpha+3)}\|K_sF\|_{L^2(0,\infty)}.
\]
After division by \eqref{eq:JNnorm},
\[
 \|B^{\alpha,\alpha+2} \|_{\ell^2 \to \ell^2}\geq \sqrt{(\alpha+2)(\alpha+3)}\frac{\|K_sF\|_{L^{2}(0,\infty)}}{\|F\|_{L^{2}(0,\infty)}}.
\]
The space $C_c(0,\infty)$ is dense in $L^2(0,\infty)$ and $K_s$ is bounded by Lemma \ref{lem:Ks-norm}. Hence, taking the supremum over such $F$, we have 
\begin{equation*}
\|B^{\alpha,\alpha+2} \|_{\ell^2 \to \ell^2}\geq \frac{\pi\sqrt{(\alpha+2)(\alpha+3)}}{\sin(\pi(\alpha+2)/2)}.
\end{equation*}

Observe that the support of $J_NF$ is contained in $[aN,bN]$.  For any $ y\in\ell^2$, we have
\[
 |\langle J_NF,y\rangle| \leq\|J_NF\|_{\ell^2} \left(\sum_{k\geq aN}|y_k|^2\right)^{1/2}\rightarrow0.
\]
Thus $J_NF\rightharpoonup0$. By the triangle inequality,
\begin{align*}
 \liminf_{N\to\infty}\|(B^{\alpha,\alpha+2}-C)J_NF\|_{\ell^2} &\geq \liminf_{N\to\infty}\|B^{\alpha,\alpha+2} J_NF\|_{\ell^2}-\lim_{N\to\infty}\|CJ_NF\|_{\ell^2}\\
 & =\liminf_{N\to\infty}\|B^{\alpha,\alpha+2} J_NF\|_{\ell^2}.
\end{align*}
Hence,
\begin{align*}
\|B^{\alpha,\alpha+2}-C\|_{\ell^2 \to \ell^2}& \geq \liminf_{N\to\infty}\frac{\|(B^{\alpha,\alpha+2}-C)J_NF\|_{\ell^2}}{\|J_NF\|_{\ell^2}}\\
&\geq \liminf_{N\to\infty}\frac{\|B^{\alpha,\alpha+2} J_NF\|_{\ell^2}}{\|J_NF\|_{\ell^2}}\\
&\geq \sqrt{(\alpha+2)(\alpha+3)}\frac{\|K_sF\|_{L^{2}(0,\infty)}}{\|F\|_{L^{2}(0,\infty)}}.
\end{align*}
Therefore, using Lemma \ref{lem:Ks-norm}, we obtain
\begin{equation*}
 \|B^{\alpha,\alpha+2} -C\|_{\ell^2 \to \ell^2}\geq \frac{\pi\sqrt{(\alpha+2)(\alpha+3)}} {\sin(\pi(\alpha+2)/2)}.
\end{equation*}
This completes the proof.
\end{proof}

Recall that the essential norm of a bounded operator $T:X\to Y$ is its distance from the compact operators. 
\begin{equation*} 
 \|T\|_{\ess, X \to Y}=\inf\{\|T-K\|:K:X\to Y\text{ is compact}\}.
\end{equation*}
Thus $\|T\|_{\ess, X \to Y}=0$ exactly when $T$ is compact.  

\begin{theorem}\label{thm:critical}
For $-1<\alpha<0$, 
\begin{equation*}\label{eq:critical-final}
 \|\HD\|_{A^2_\alpha\to A^2_{\alpha+2}} =\|\HD\|_{\mathrm e,A^2_\alpha\to A^2_{\alpha+2}} =\frac{\pi\sqrt{(\alpha+2)(\alpha+3)}}
 {-\sin(\pi\alpha/2)}.
\end{equation*}
\end{theorem}

\begin{proof}
From the proof of Theorem \ref{thm:bdd}, we obtain
\begin{equation*}
\|\mathcal H_2\|_{A_{\alpha}^{2}\to A_{\alpha+2}^{2}}=\|B^{\alpha,\alpha+2}\|_{\ell^2 \to \ell^2}, \qquad
\|\mathcal H_2\|_{\mathrm e,A^2_\alpha\to A^2_{\alpha+2}}=\|B^{\alpha,\alpha+2}\|_{\mathrm e,\ell^2 \to \ell^2}.
\end{equation*} 

The factorization $\mathcal H_2=D\mathcal H$, together with the sharp norm of $\mathcal H$ from Lemma \ref{lemma:Hnorm} and the exact critical norm of $D$ from Lemma \ref{lem:D}, gives
\begin{equation*}
\|\mathcal H_2\|_{A_{\alpha}^{2}\to A_{\alpha+2}^{2}}\leq \|D\|_{A_{\alpha}^{2}\to A_{\alpha+2}^{2}}\,\|\mathcal H\|_{A_{\alpha}^{2}\to A_{\alpha}^{2}}=\frac{\pi\sqrt{(\alpha+2)(\alpha+3)}}
{\sin(\pi(\alpha+2)/2)}.
\end{equation*}

By Lemma \ref{lem:discrete-continuous}, for every compact operator $C$ on the   space $\ell^2$,
\begin{equation*}
\|B^{\alpha,\alpha+2}-C\|_{\ell^2 \to \ell^2}\geq\frac{\pi\sqrt{(\alpha+2)(\alpha+3)}}{\sin(\pi(\alpha+2)/2)}.
\end{equation*}
Taking the infimum over all compact operators $C$ yields
\begin{equation*}
\|B^{\alpha,\alpha+2} \|_{\mathrm e,\ell^2 \to \ell^2}\geq \frac{\pi\sqrt{(\alpha+2)(\alpha+3)}}{\sin(\pi(\alpha+2)/2)}.
\end{equation*}
Therefore, we obtain
\begin{equation*}
\|\mathcal H_2\|_{\mathrm e,A^2_\alpha\to A^2_{\alpha+2}}\geq \frac{\pi\sqrt{(\alpha+2)(\alpha+3)}}{\sin(\pi(\alpha+2)/2)}.
\end{equation*}

Since the essential norm never exceeds the operator norm, that is
\[
\|\mathcal H_2\|_{\mathrm e,A^2_\alpha\to A^2_{\alpha+2}}
\leq\|\mathcal H_2\|_{A_{\alpha}^{2}\to A_{\alpha+2}^{2}}\leq \frac{\pi\sqrt{(\alpha+2)(\alpha+3)}}{\sin(\pi(\alpha+2)/2)},
\]
and hence
\[
\|\mathcal H_2\|_{\mathrm e,A^2_\alpha\to A^2_{\alpha+2}}=\|\mathcal H_2\|_{A_{\alpha}^{2}\to A_{\alpha+2}^{2}}
=\frac{\pi\sqrt{(\alpha+2)(\alpha+3)}}{\sin(\pi(\alpha+2)/2)}.
\]
It is clear that 
$$
\frac{\pi\sqrt{(\alpha+2)(\alpha+3)}}{\sin(\pi(\alpha+2)/2)}=\frac{\pi\sqrt{(\alpha+2)(\alpha+3)}}
 {-\sin(\pi\alpha/2)}.
$$
The proof is complete.
\end{proof}

\begin{remark}  It is worth noting that the equality of the operator norm and
essential norm does not follow merely from the corresponding behavior of $\HH$. 
\end{remark}

\begin{proposition}\label{prop:Gram}
Let $-1<\alpha<0$ and $\beta>\alpha+2$. Define $G^{\alpha,\beta}=(G_{k\ell}^{\alpha,\beta})_{k,\ell\geq0}$ by
\begin{equation}\label{eq:Gram}
 G_{k\ell}^{\alpha,\beta} =\frac1{\sqrt{w_k(\alpha)w_\ell(\alpha)}} \sum_{n=0}^\infty \frac{(n+1)^2w_n(\beta)} {(n+k+1)(n+\ell+1)}.
\end{equation}
Then $G^{\alpha,\beta}$ is the matrix of $\HD^*\HD$ on
$A^2_\alpha$ and
\begin{equation}\label{eq:norm-lambda}
 \|\HD\|_{A^2_\alpha\to A^2_\beta} =\sqrt{\lambda_{\max}(G^{\alpha,\beta})}.
\end{equation}
Equivalently,
\begin{equation}\label{eq:Gram-integral}
 G_{k\ell}^{\alpha,\beta} =\frac1{\sqrt{w_k(\alpha)w_\ell(\alpha)}} \int_0^1\int_0^1x^ky^\ell F_\beta(xy)\,dx\,dy,
\end{equation}
where
\begin{equation*} 
 F_\beta(r)=\sum_{n=0}^\infty(n+1)^2w_n(\beta)r^n.
\end{equation*}
\end{proposition}

\begin{proof} Let
\[ B^{\alpha,\beta}= U_\beta\mathcal H_2U_\alpha^{-1} \]
be the matrix representation of $\mathcal H_2$ on $\ell^2$, where the unitary maps $U_\alpha$ and $U_\beta$ are defined in the proof of Theorem \ref{thm:bdd}.
Denote by $\{e_k\}_{k\geq0}$ the standard orthonormal basis of $\ell^2$. We have
\begin{align*}
\langle \mathcal{H}_{2}^*\mathcal{H}_{2}e_{\ell}^{(\alpha)},e_{k}^{(\alpha)}\rangle_{A_{\alpha}^2} &= \langle \mathcal{H}_{2}^*\mathcal{H}_{2}U_{\alpha}^{-1}e_\ell,U_{\alpha}^{-1}e_k\rangle_{A_{\alpha}^2}
 =\langle \mathcal{H}_{2}e_\ell^{(\alpha)},\mathcal{H}_{2}e_k^{(\alpha)} \rangle_{A_{\beta}^2}\\
&=\langle B^{\alpha,\beta}e_\ell,B^{\alpha,\beta}e_k\rangle_{\ell^2} =\sum_{n=0}^\infty b_{n\ell}^{\alpha,\beta}b_{nk}^{\alpha,\beta}.
\end{align*} 
Since $B^{\alpha,\beta} e_k,B^{\alpha,\beta} e_\ell \in \ell^2$, the Cauchy–Schwarz inequality implies that the above series converges. Since $U_\alpha$ and $U_\beta$ are unitary maps, the operator $G^{\alpha,\beta}=(B^{\alpha,\beta})^*B^{\alpha,\beta}$ is the matrix representation of $\mathcal H_2^*\mathcal H_2$ on $A_\alpha^2$. Substituting \eqref{eq:Bmatrix} into the last series  and comparing coefficients, we obtain \eqref{eq:Gram}. 

Since $-1<\alpha<0$ and $\beta>\alpha+2$, Theorem \ref{thm:bdd} shows that $B^{\alpha,\beta}$ is compact.  Hence $(B^{\alpha,\beta})^*B^{\alpha,\beta}$ is a positive compact self-adjoint operator.  By the spectral theorem, its nonzero spectrum consists of positive eigenvalues converging to zero if infinite in number.  Since $B^{\alpha,\beta} \ne0$,
the number $\|(B^{\alpha,\beta})^*B^{\alpha,\beta}\|_{\ell^2 \to \ell^2}$ is its largest eigenvalue.  Therefore
\[
 \|\HD\|_{A_{\alpha}^2\to A_{\beta}^2}^2=\|B^{\alpha,\beta} \|_{\ell^2 \to \ell^2}^2=\|(B^{\alpha,\beta})^*B^{\alpha,\beta}\|_{\ell^2 \to \ell^2}
 =\lambda_{\max}(G^{\alpha,\beta}),
\]
which proves \eqref{eq:norm-lambda}.

For the integral form, use the elementary identity
\[
 \frac1{(n+k+1)(n+\ell+1)} =\int_0^1\int_0^1x^{n+k}y^{n+\ell}\,dx\,dy
\]
in each summand of \eqref{eq:Gram}.  All summands are nonnegative for $x,y\in(0,1)$, so Tonelli's theorem permits the interchange of the sum
and the two integrals.  The sum remaining inside is
\[
 \sum_{n=0}^\infty(n+1)^2w_n(\beta)(xy)^n=F_\beta(xy),
\]
which is exactly \eqref{eq:Gram-integral}.
\end{proof}

\begin{remark} The   norm formula also describes the behavior at the two endpoints
of the parameter range $-1<\alpha<0$. As $\alpha\downarrow-1$, although $A_{-1}^2$ is not included, the coefficient weights satisfy $w_n(\alpha)\longrightarrow1$, which corresponds to the $H^2$ coefficient norm. Hence, by Theorem \ref{thm:critical},
\[
\lim_{\alpha\downarrow-1}\|\mathcal H_2\|_{A_\alpha^2\to A_{\alpha+2}^2} =\frac{\pi\sqrt{2}}{\sin(\pi/2)}=\pi\sqrt{2}.
\]
Thus the formal endpoint   is the mapping $\mathcal H_2:H^2\to A_1^2$ with limiting norm $\pi\sqrt{2}$.

As $\alpha\uparrow0$,
\[
-\sin\frac{\pi\alpha}{2}\asymp -\frac{\pi\alpha}{2}, \qquad \sqrt{(\alpha+2)(\alpha+3)}\longrightarrow\sqrt{6},
\]
and therefore
\[
\|\mathcal H_2\|_{A_\alpha^2\to A_{\alpha+2}^2} \asymp \frac{2\sqrt{6}}{-\alpha}.
\]
Hence the norm diverges as $\alpha\uparrow0$. This divergence is consistent with the failure of the functional
$
f\longmapsto\int_0^1 f(t)\,dt
$
to be bounded at $\alpha=0$, since this functional is precisely the constant coefficient $(\mathcal H_2f)(0)$.
\end{remark}\medskip

\section{Schatten classes and Schatten norms}
\label{sec:Schatten}

In this section, we prove a statement stronger than compactness.  The proof uses a
rank-one decomposition. This decomposition also shows that if $\beta$ is just slightly larger than $\alpha+2$, then the operator is already trace class. For further results on   Schatten class of general Hilbert-type operator, see \cite{PelaezSeco2018}.

For $0<t<1$, define 
\begin{equation*} 
 v_t=\left(\frac{t^k}{\sqrt{w_k(\alpha)}}\right)_{k\geq0},  \qquad u_t=\left((n+1)\sqrt{w_n(\beta)}t^n\right)_{n\geq0}. \nonumber 
\end{equation*}
We use the convention
\[
 (u_t\otimes v_t)x=\langle x,v_t\rangle_{\ell^2}u_t.
\]
For polynomials, the identity
\[
 \int_0^1(n+1)\sqrt{\frac{w_n(\beta)}{w_k(\alpha)}}  t^{n+k}\,dt=b_{nk}^{\alpha,\beta}
\]
holds for every entry. Therefore 
\begin{equation*} 
 B^{\alpha,\beta}=\int_0^1u_t\otimes v_t\,dt.  
\end{equation*}

\begin{lemma}\label{lem:uvnorms}
Let $-1<\alpha<0$ and $\beta>\alpha+2$.  Then
\begin{equation}\label{eq:nuclear-integral}
 \int_0^1\|u_t\|_{\ell^2}\|v_t\|_{\ell^2}\,dt<\infty.
\end{equation}
\end{lemma}

\begin{proof}
Since $t<1$, all series below converge absolutely.  Using \eqref{eq:wnexact} in Lemma \ref{lem:monomial},  a simple computation gives
\[
 \|v_t\|_{\ell^2}^2 =\sum_{k=0}^\infty\frac{t^{2k}}{w_k(\alpha)} =\sum_{k=0}^\infty\frac{(\alpha+2)_k}{k!}t^{2k} =(1-t^2)^{-\alpha-2}.
\]
By direct computation,
\begin{equation*} 
 \|u_t\|_{\ell^2}^2=F_\beta(t^2).
\end{equation*}
For all $n\geq1$, by \eqref{eq:wn-asymp}, we have 
\begin{equation*} 
(n+1)^2w_n(\beta) \asymp (n+1)^{1-\beta}.
\end{equation*}
In order to prove \eqref{eq:nuclear-integral}, we only need to consider the behavior of $F_\beta(r)$ near the boundary. Put $\delta=1-r$. 

 Assume $1<\beta<2$. Since $r^n=(1-\delta)^n =\exp\bigl(n\log(1-\delta)\bigr)$ and $\log(1-\delta)\geq-C\delta$ for sufficiently small $\delta$, it follows that $r^n\geq e^{-C}>0$ uniformly for $1\leq n\leq  \delta^{-1}$. Therefore,
\begin{equation*}
F_\beta(r) \gtrsim \sum_{1\leq n\leq\delta^{-1} }n^{1-\beta} \gtrsim \int_1^{\delta^{-1}}x^{1-\beta}\,dx \asymp \delta^{\beta-2}.
\end{equation*}
For the upper estimate, use $r^n\leq e^{-n(1-r)}$, we obtain
\[
F_\beta(r)\lesssim\sum_{n=1}^{\infty}n^{1-\beta}e^{-n\delta}.
\]
For $n\ge 1$ and for all $x\in [n-1,n]$, we have $n^{1-\beta}e^{-n\delta}\le x^{1-\beta}e^{-x\delta}$. Therefore,
\[
F_\beta(r)\lesssim\sum_{n=1}^{\infty}n^{1-\beta}e^{-n\delta}\leq\int_0^\infty x^{1-\beta}e^{-\delta x}\,dx=\Gamma(2-\beta)\delta^{\beta-2}.
\]
For  $\beta=2$, we have
\begin{align*}
F_2(r)=1+\sum_{n=1}^{\infty}(n+1)^2w_n(2)r^n&=1+\sum_{n=1}^{\infty}\frac{6(n+1)}{(n+2)(n+3)}r^n\\
&\asymp 1+\sum_{n=1}^{\infty}\frac{r^n}{n}=\log\frac{e}{1-r}.
\end{align*}
For $\beta>2$, we have
\[
F_\beta(r)=\sum_{n=0}^{\infty}(n+1)^2w_n(\beta)r^n\lesssim\sum_{n=0}^{\infty}(n+1)^{1-\beta}r^n  \le \sum_{n=0}^{\infty}(n+1)^{1-\beta}.
\]
Hence,
\[
F_\beta(1)=\sum_{n=0}^{\infty}(n+1)^2w_n(\beta)<\infty .
\]
Moreover, all coefficients of $F_\beta(r)$ are nonnegative, and hence $F_\beta(r)$ is increasing in $r$. By the monotone convergence theorem,
\[
\lim_{r\rightarrow1^-}F_\beta(r)=\sum_{n=0}^{\infty}(n+1)^2w_n(\beta)=F_\beta(1).
\]
Since the first term equals $1$, we have $F_\beta(r)\ge 1$ and therefore
$
F_\beta(r)\asymp1$ as $r\rightarrow1^- .$ Consequently, as $r\to1^-$,
\begin{equation*} 
 F_\beta(r)\asymp
 \begin{cases}
  (1-r)^{\beta-2},&1<\beta<2,\\
  \log\dfrac{e}{1-r},&\beta=2,\\
  1,&\beta>2.
 \end{cases}
\end{equation*}

On every interval $[0,t_0]$ with $t_0<1$, both norms in \eqref{eq:nuclear-integral} are bounded and
continuous, so only the limit $t\to1^-$  needs to be considered.  Since $1-t^2\asymp1-t$, if $1<\beta<2$ the integrand is comparable near one to
\[
 (1-t)^{(\beta-2)/2}(1-t)^{-(\alpha+2)/2} =(1-t)^{(\beta-\alpha-4)/2},
\]
which is integrable exactly when $\beta>\alpha+2$.  If $\beta=2$, the integrand is comparable   to 
\[
 (1-t)^{-(\alpha+2)/2} \left(\log\frac{e}{1-t}\right)^{1/2}
\]
up to fixed constants. Since $(\alpha+2)/2<1$, the expression is integrable, and a logarithmic factor does not change a strict power inequality.  If $\beta>2$, the square root of $F_\beta(t^2)$ stays bounded and the same power test applies.  This proves \eqref{eq:nuclear-integral} in all cases.
\end{proof}

\begin{theorem}\label{thm:trace} For $\alpha,\beta>-1$, the following statements are equivalent. 
\begin{enumerate}
\item $\HD:A^2_\alpha\to A^2_\beta$ is trace class;
\item $\HD\in\Sp_p(A^2_\alpha,A^2_\beta)$ for every $1\leq p<\infty$;
\item $-1<\alpha<0$ and $\beta>\alpha+2$.
\end{enumerate}
\end{theorem}

\begin{proof} $(1)\Rightarrow(2)$. It is clear.

$(2)\Rightarrow(3)$.  If (2) holds, then in particular $\HD\in\Sp_2$.  Every Schatten operator with finite
exponent is compact, so the compactness condition in Theorem \ref{thm:bdd} forces $-1<\alpha<0$ and $\beta>\alpha+2$.

$(3)\Rightarrow (1)$. Since  $u_t\otimes v_t$ is a rank-one operator, we get 
\begin{equation}\label{eq:rank-one-trace-norm}
 \|u_t\otimes v_t\|_{\Sp_1}=\|u_t\|_{\ell^2}\|v_t\|_{\ell^2}.
\end{equation}
We first justify strong measurability, which is needed for the Bochner integral. Fix $0<r<1$. For $s,t\in[0,r]$, we first consider the vector-valued map
$t\mapsto v_t$. By definition,
\[
\|v_t-v_s\|_{\ell^2}^{2}=\sum_{n=0}^{\infty}\frac{|t^n-s^n|^2}{w_n(\alpha)} .
\]
For every fixed $n$, $|t^n-s^n|\to 0$ as $t\to s$. Moreover, 
\[
\frac{|t^n-s^n|^2}{w_n(\alpha)}\leq\frac{4r^{2n}}{w_n(\alpha)} .
\]
and the series $\sum_{n=0}^{\infty}\frac{4r^{2n}}{w_n(\alpha)}$ is convergent. Therefore, the dominated convergence theorem for series yields $\|v_t-v_s\|_{\ell^2}\rightarrow0$ as $t\to s$.
Similarly,
\[
\|u_t-u_s\|_{\ell^2}^{2}=\sum_{n=0}^{\infty}(n+1)^2w_n(\beta)|t^n-s^n|^2 .
\]
Since
\[
(n+1)^2w_n(\beta)|t^n-s^n|^2\leq4(n+1)^2w_n(\beta)r^{2n}
\]
and
\[
\sum_{n=0}^{\infty}(n+1)^2w_n(\beta)r^{2n}=F_\beta(r^2)<\infty ,
\]
another application of dominated convergence gives
$
\|u_t-u_s\|_{\ell^2}\rightarrow0  
$
as $t\to s$. Hence both maps $t\mapsto v_t$ and $t\mapsto u_t$ are norm-continuous on $[0,r]$, and therefore on $[0,1)$.

For rank-one operators,
\begin{align*}
 \|u_t\otimes v_t-u_s\otimes v_s\|_{\Sp_1}&=\|(u_t-u_s)\otimes v_t+u_s\otimes (v_t-v_s)\|_{\Sp_1}\\
 &\leq\|(u_t-u_s)\otimes v_t\|_{\Sp_1}
      +\|u_s\otimes(v_t-v_s)\|_{\Sp_1}\\
 &=\|u_t-u_s\|_{\ell^2}\,\|v_t\|_{\ell^2}+\|u_s\|_{\ell^2}\,\|v_t-v_s\|_{\ell^2}.
\end{align*}
Thus $t\mapsto u_t\otimes v_t$ is norm-continuous from $[0,1)$ into the separable Banach space $\Sp_1$, and in particular it is strongly
measurable.  By Lemma \ref{lem:uvnorms} and \eqref{eq:rank-one-trace-norm}, its $\Sp_1$-norm is integrable.  Therefore
\[
 T=\int_0^1u_t\otimes v_t\,dt
\]
exists as a Bochner integral in the Banach space $\Sp_1$, and the triangle inequality gives
\[
 \|T\|_{\Sp_1} \leq\int_0^1\|u_t\|\|v_t\|\,dt.
\]

We verify that $T$ is precisely the matrix of $\HD$.  Let $\left \{ e_k \right \}_{k\ge 0} $ be standard coordinate vectors in $\ell^2$.  Since taking a matrix entry is a bounded
linear functional on $\Sp_1$, it may be passed through the Bochner integral.  Thus
\begin{align*}
 \langle Te_k,e_n\rangle_{\ell^2}  &=\int_0^1   \langle (u_t\otimes v_t)e_k,e_n\rangle_{\ell^2}\,dt\\
 &=\int_0^1(n+1)   \sqrt{\frac{w_n(\beta)}{w_k(\alpha)}}t^{n+k}\,dt\\
 &=\frac{n+1}{n+k+1}   \sqrt{\frac{w_n(\beta)}{w_k(\alpha)}}
 =b_{nk}^{\alpha,\beta}.
\end{align*}
Hence $T$ and $B^{\alpha,\beta}$ agree on every finitely supported sequence.  Both are bounded, and since these sequences are dense in $\ell^2$,
we have $T=B^{\alpha,\beta}$.  This proves that $\HD\in\Sp_1$. The proof is complete.
\end{proof}

\begin{corollary}\label{prop:Sq-exact}
Assume $-1<\alpha<0$ and $\beta>\alpha+2$. $G^{\alpha,\beta} =(B^{\alpha,\beta})^*B^{\alpha,\beta}$ is a positive compact operator on $\ell^2$ and is unitarily equivalent to $\mathcal H_2^*\mathcal H_2$.  Let
\begin{equation*}
\lambda_1(G^{\alpha,\beta})\geq\lambda_2(G^{\alpha,\beta})\geq\cdots\geq0
\end{equation*}
denote its eigenvalues, repeated according to multiplicity. Then, for every
$1\leq p<\infty$,
\begin{equation}\label{eq:Sq-spectral-exact}
\|\mathcal H_2\|_{\Sp_p}^p=\sum_{j=1}^{\infty}\lambda_j(G^{\alpha,\beta})^{p/2}.
\end{equation}
\end{corollary}

\begin{proof} By definition, the singular values of $\HD$ are the eigenvalues of $|\HD|=(\HD^*\HD)^{1/2}$, counted with multiplicity.  The continuous
functional calculus for the positive compact operator $G^{\alpha,\beta}$ shows that the eigenvalues of its square root are
$\lambda_j(G^{\alpha,\beta})^{1/2}$.  Hence
\[
 s_j(\HD)=\lambda_j(G^{\alpha,\beta})^{1/2}.
\]
The definition $\|\HD\|_{\Sp_p}^p=\sum_js_j(\HD)^p$ now gives \eqref{eq:Sq-spectral-exact}.  Convergence for every $p\geq1$ follows
from Theorem \ref{thm:trace}. 
\end{proof}

Next, we compute the Hilbert--Schmidt norm of the operator $\mathcal{H}_2:A_{\alpha}^2\to A_{\beta}^2$.

\begin{theorem}\label{thm:HS}
Let $-1<\alpha<0$ and $\beta>\alpha+2$. Then the Hilbert--Schmidt norm of the operator $\mathcal{H}_2:A_{\alpha}^2\to A_{\beta}^2$ is given by
\begin{equation}\label{eq:HS-double}
 \|\HD\|_{\Sp_2}^2 =\sum_{n,k\geq0} \frac{(n+1)^2}{(n+k+1)^2} \frac{\Gamma(\beta+2)\Gamma(n+1)}{\Gamma(n+\beta+2)}
 \frac{\Gamma(k+\alpha+2)} {\Gamma(\alpha+2)\Gamma(k+1)}.
\end{equation}
Equivalently,
\begin{equation*} 
 \|\HD\|_{\Sp_2}^2 =\int_0^1(-\log r)(1-r)^{-\alpha-2}F_\beta(r)\,dr.
\end{equation*}
\end{theorem}

\begin{proof} Since $-1<\alpha<0$ and $\beta>\alpha+2$,  it follows from Theorem \ref{thm:trace} that $\mathcal{H}_{2}$ is Hilbert--Schmidt. Let $\left \{ e_k^{(\alpha)} \right \}_{k\ge 0} $ be the orthonormal basis in \eqref{eq:orthonormal-basis}.  By the definition of the Hilbert--Schmidt norm and Parseval's identity, we have
\begin{align*}
 \|\HD\|_{\Sp_2}^2 &=\sum_{k=0}^\infty   \|\HD e_k^{(\alpha)}\|_{A^2_\beta}^2=\sum_{k=0}^\infty\sum_{n=0}^\infty   |b_{nk}^{\alpha,\beta}|^2\\
 &=\sum_{n,k\geq0}   \frac{(n+1)^2}{(n+k+1)^2}   \frac{w_n(\beta)}{w_k(\alpha)}.
\end{align*}
Since all terms involved are nonnegative, Tonelli's theorem allows us to change the order of summation without affecting the value of the series.  Substitution \eqref{eq:wnexact} into the above expression gives \eqref{eq:HS-double}.

It is easy to check that 
\[
 \frac1{(n+k+1)^2}=\int_0^1r^{n+k}(-\log r)\,dr.
\] 
 We substitute this identity into the Hilbert--Schmidt norm. Since all terms are nonnegative, by Tonelli’s theorem,
\begin{align*}
\|\mathcal H_2\|_{\Sp_2}^2 &=\int_0^1(-\log r)\left(\sum_{k=0}^{\infty}\frac{r^k}{w_k(\alpha)}\right)\left(\sum_{n=0}^{\infty}(n+1)^2w_n(\beta)r^n\right)\,dr \\
&=\int_0^1(-\log r)(1-r)^{-\alpha-2}F_\beta(r)\,dr .
\end{align*}
This completes the proof.
\end{proof}

\bigskip

\noindent\textbf{Data Availability.} Data sharing is not applicable for this article as no datasets were generated or analyzed during the current study.
\medskip

 \noindent\textbf{Funding.}   Not applicable.
 
 \bigskip

\noindent\textbf{Conflict of interest.} The authors declare no competing interests.

\end{document}